\documentclass[12pt,reqno]{amsart}
\usepackage{amsmath,amssymb,comment,color,verbatim}
\usepackage[pagebackref,pdftex,hyperindex]{hyperref}

\usepackage[margin=3cm]{geometry}
\usepackage{color}

\newcommand{\C}{\mathbb{C}}

\newcommand{\R}{\mathbb{R}}

\newcommand{\cC}{\mathcal{C}}

\newcommand{\cF}{\mathcal{F}}

\newcommand{\cH}{\mathcal{H}}

\newcommand{\cJ}{\mathcal{J}}

\newcommand{\cL}{\mathcal{L}}

\newcommand{\cS}{\mathcal{S}}

\newcommand{\cV}{\mathcal{V}}

\renewcommand{\a}{\alpha}
\renewcommand{\b}{\beta}

\renewcommand{\d}{\delta}
\newcommand{\e}{\varepsilon}
\newcommand{\s}{\sigma}

\renewcommand{\r}{\rho}

\newcommand{\sphi}{\underline{\phi}}
\newcommand{\tlambda}{\tilde{\lambda}}
\newcommand{\bnabla}{\overline{\nabla}}

\newcommand{\hpi}{\frac{\pi}{2}}
\newcommand{\p}{\partial}

\newcommand{\dd}{\sqrt{-1}\partial \bar{\partial}}
\newcommand{\ddt}{\frac{d}{dt}}
 
\newcommand{\cf}{{\rm cf.\ }} 
\newcommand{\eg}{{\rm e.g.\ }} 
\newcommand{\ie}{{\rm i.e.\ }} 

\renewcommand{\Re}{\mathrm{Re}}
\renewcommand{\Im}{\mathrm{Im}}

\DeclareMathOperator{\Id}{Id}

\DeclareMathOperator{\osc}{osc}
\DeclareMathOperator{\Rm}{Rm}

\DeclareMathOperator{\SLag}{SLag}

\DeclareMathOperator{\Tr}{Tr}

\renewcommand{\leq}{\leqslant}
\renewcommand{\geq}{\geqslant}

\renewcommand{\hat}{\widehat}
\renewcommand{\tilde}{\widetilde}

\numberwithin{equation}{section}       

\newtheorem{prop} {Proposition} [section]
\newtheorem{thm}[prop] {Theorem} 
\newtheorem{dfn}[prop] {Definition}
\newtheorem{lem}[prop] {Lemma}
\newtheorem{cor}[prop]{Corollary}

\theoremstyle{remark}
\newtheorem*{ackn}{\bf{Acknowledgment}} 
\newtheorem{rk}[prop]{Remark}

\title[Tan-concavity and the tangent Lagrangian phase flow]{Tan-concavity property for Lagrangian phase operators and applications to \\the tangent Lagrangian phase flow} 
\date{\today}

\author[R. Takahashi]{Ryosuke Takahashi}
\address{Faculty of Mathematics\\
 Kyushu University\\
744\\
Motooka\\
Nishi-ku\\
Fukuoka\\
819-0395\\
 JAPAN}
\email{rtakahashi@math.kyushu-u.ac.jp}

\subjclass[2010]{Primary 53C55; Secondary 53C44}
\keywords{deformed Hermitian Yang-Mills metric, parabolic PDE, Thomas-Yau conjecture}

\begin{document}
\maketitle
\begin{abstract}
We explore the tan-concavity of the Lagrangian phase operator for the study of the deformed Hermitian Yang-Mills (dHYM) metrics. This new property compensates for the lack of concavity of the Lagrangian phase operator as long as the metric is almost calibrated. As an application, we introduce the tangent Lagrangian phase flow (TLPF) on the space of almost calibrated $(1,1)$-forms that fits into the GIT framework for dHYM metrics recently discovered by Collins-Yau. The TLPF has some special properties that are not seen for the line bundle mean curvature flow (\ie the mirror of the Lagrangian mean curvature flow for graphs). We show that the TLPF starting from any initial data exists for all positive time. Moreover, we show that the TLPF converges smoothly to a dHYM metric assuming the existence of a $C$-subsolution, which gives a new proof for the existence of dHYM metrics in the highest branch.
\end{abstract}
\tableofcontents
\section{Introduction}
Let $X$ be a compact $n$-dimensional complex manifold with a fixed K\"ahler form $\a$ and a closed real $(1,1)$-form $\hat{\chi}$. For $\phi \in C^\infty(X;\R)$, we set $\chi:=\hat{\chi}+\dd \phi$ (where the forms $\hat{\chi}$, $\chi$ are not necessary K\"ahler).
We say that $\phi \in C^\infty(X,\R)$ is a deformed Hermitian-Yang Mills (dHYM) metric if it satisfies
\begin{equation} \label{dHYM}
\Im \big(e^{-\sqrt{-1} \hat{\Theta}}(\a+\sqrt{-1} \chi)^n \big)=0.
\end{equation}
We define a topological invariant
\[
Z:=\int_X (\a+\sqrt{-1} \chi)^n
\]
and assume $Z \neq 0$. Then the constant angle $\hat{\Theta}$ is uniquely determined (mod. $2\pi$) by the property that
\[
Z e^{-\sqrt{-1} \hat{\Theta}} \in \R_{>0}.
\]
The equation \eqref{dHYM} first appeared in the physics literature \cite{MMMS00}, and \cite{LYZ01}  from mathematical side as the mirror object to a special Lagrangian in the setting of semi-flat mirror symmetry. More specifically, let $X, W$ be a mirror Calabi-Yau pair defined as dual torus fibrations over a base tori $B$. Then in this setting, a Lagrangian section $s$ of the torus fibration $W \to B$ corresponds to a fiber metric $\phi$ on a holomorphic line bundle $L \to X$. If $s$ is a special Lagrangian section, then it further corresponds to a fiber metric $\phi$ satisfying the dHYM equation \eqref{dHYM} via the Fourier-Mukai transform. Recently, the dHYM metric has been studied actively (\eg \cite{CCL20,Che19,CJY15,CY18,HJ20,HY19,JY17,Pin19,SS19}). We define
\[
\theta(\lambda):=\sum_{i=1}^n \arctan \lambda_i,
\]
where $\arctan$ takes values in $(-\hpi,\hpi)$ so the image of $\theta$ lies in $(-n \hpi,n \hpi)$. For an $n \times n$ Hermitian matrix $A$ with eigenvalues $\lambda[A]$, we set
\[
\Theta(A):=\theta(\lambda[A]).
\]
The function $\Theta$ is smooth by the symmetry of $\theta$. For $\phi \in C^\infty(X;\R)$, we set $A[\phi]:=\chi_{i\bar{j}}\a^{k\bar{j}}$. This is a Hermitian endomorphism on $T^{1,0}X$ with respect to $\a$. We denote the eigenvalues of $A[\phi]$ by $\lambda[\phi]$. Then according to the argument \cite{JY17}, the condition \eqref{dHYM} is equivalent to
\begin{equation} \label{dHYMf}
\Theta(A[\phi])=\hat{\Theta} \quad (\text{mod. $2\pi$}),
\end{equation}
where we regard $A[\phi]$ as an $n \times n$ Hermitian matrix at each point by taking normal coordinates. We say that $\phi \in C^\infty(X;\R)$ is {\it supercritical} (resp. {\it hypercritical}) if it satisfies $\Theta(A[\phi])>(n-2)\hpi$ (resp. $>(n-1)\hpi$). In particular, since $\arctan(\cdot)$ takes values in $(-\hpi,\hpi)$, the condition $\Theta(A[\phi])>(n-1)\hpi$ yields that all the components of $\lambda[\phi]$ are positive, and hence $\chi_\phi$ is K\"ahler. Although the equation \eqref{dHYMf} is elliptic, it has several problems on analysis. Most seriously, the operator $\Theta$ fails to be concave in general. One can easily check by a straightforward computation that $\Theta(A[\phi])$ is concave if and only if $\phi$ is hypercritical. The concavity of the operator is essential to apply the Evans-Krylov theory \cite{Kry82,Wan12} for $C^{2,\b}$ estimate. To deal with this problem, an important observation shown by Yuan \cite{Yua05} is that the level set $\{\lambda \in \R^n|\theta(\lambda)=\s\}$ is still convex as long as $\s \geq (n-2)\hpi$. This result indicates that one may get a concave function $\varsigma \circ \theta$ by composing $\theta$ with a sufficiently concave function $\varsigma \colon \R \to \R$. Actually Collins-Picard-Wu \cite{CPW17} showed that the function $-e^{-A\theta(\lambda)}$ is concave for a large enough constant $A=A(\d)$ as long as $\lambda$ satisfies $\theta(\lambda) \geq (n-2)\hpi+\d$ for some $\d>0$.

In this paper, we consider yet another choice of $\varsigma$. Let us consider the function
\[
f(\lambda):=\tan(\theta(\lambda)-\hat{\Theta}).
\]
The function $f$ is not globally defined, but it is well-defined restricted to a subset of $\R^n$;
\[
\cS:=\bigg\{\lambda \in \R^n \bigg| |\theta(\lambda) -\hat{\Theta}|<\hpi \bigg\}.
\]
Compared to the result \cite{CPW17}, one might find it strange that our choice of $\varsigma(x)=\tan(x-\hat{\Theta})$ is not concave, so it is not sure that the composition $f=\varsigma \circ \theta$ is also a concave function. In our first main theorem, we check the concavity property for $f$;
\begin{thm} \label{tanconv}
Assume that $\hat{\Theta} \in ((n-1)\hpi,n \hpi)$, Then the function $\theta-\hat{\Theta}$ is tan-concave, \ie the composition $f=\tan \circ (\theta-\hat{\Theta})$ is concave on $\cS$.
\end{thm}

After the author had posted the preprint on arXiv, he was informed by F.~R.~Harvey and H.~B.~Lawson, Jr. that Theorem \ref{tanconv} has a relation to their recent work \cite{HL20}, in which they proved a ``tameness'' condition for the composition $\tan(\SLag/n)$ to solve the inhomogeneous Dirichlet problem $\SLag(D^2u):=\Tr(\arctan(D^2u))=\psi$ for a function $u \colon \overline{\Omega} \to \R$, where $\Omega \subset \R^n$ is a bounded domain with smooth strictly convex boundary $\p \Omega$, and $\psi$ is a continuous function on $\overline{\Omega}$ satisfying $\psi(\overline{\Omega}) \subset ((n-2)\hpi,n\hpi)$ (see \cite[Section 5]{HL20} for more details).

We expect that the tan-concavity property is useful for the study of dHYM metrics or minimal Lagrangian graphs. As a demonstration, we provide a new geometric flow approach to construct dHYM metrics. Associated to the set $\cS \subset \R^n$, we define the space of {\it almost calibrated} potential functions (\cf \cite{CCL20});
\[
\cH:=\bigg\{ \phi \in C^\infty(X;\R) \bigg| |\Theta(A[\phi])-\hat{\Theta}|< \hpi \bigg\}.
\]
We remark that the set $\cH$ is strictly contained in the space of potentials with supercritical phase when $\hat{\Theta}>(n-1)\hpi$. In later arguments, we always assume that $\cH$ is not empty and $0 \in \cH$ by replacing a reference form $\hat{\chi}$. For any $\phi_0 \in \cH$, we define
\begin{equation} \label{TLPF}
\ddt \phi_t=\tan(\Theta(A[\phi_t])-\hat{\Theta}).
\end{equation}
We have $f_i=\frac{1+f^2}{1+\lambda_i^2}>0$ for all $i$, which guarantees the ellipticity of the operator in the RHS. So the short time existence follows from general theory. We would like to call \eqref{TLPF}, the {\it tangent Lagrangian phase flow} (TLPF). For simplicity, we set $F(A):=f(\lambda[A])$ so that the flow equation is given by $\ddt \phi=F(A[\phi])$. On the other hand, Jacob-Yau \cite{JY17} introduced the {\it line bundle mean curvature flow} (LBMCF)
\begin{equation} \label{LBMCF}
\ddt \phi_t=\Theta(A[\phi_t])-\hat{\Theta}
\end{equation}
as the mirror of the Lagrangian mean curvature flow (LMCF) for graphs. In a formal level, the TLPF is similar to the LBMCF whenever the Lagrangian phase $\Theta(A[\phi_t])$ is very close to $\hat{\Theta}$. However it is expected that the limiting behavior of these two flows are quite different. A similar observation can be found in the comparison of the K\"ahler-Ricci flow and the inverse Monge-Amp\`ere flow on Fano manifolds \cite{CHT17}. For the moment, we assume $\hat{\Theta} \in ((n-1)\hpi,n \hpi)$. Then the virtue of the TLPF is the following;
\begin{itemize}
\item The TLPF has some special properties that are not seen for the LBMCF (\cf Remark \ref{cff} and Remark \ref{efc}).
\item Our choice of $\varsigma(x)=\tan(x-\hat{\Theta})$ has a natural geometric meaning; the TLPF perfectly fits into the GIT framework recently discovered by Collins-Yau \cite{CY18} (as the mirror of Solomon's \cite{Sol13} and Thomas's \cite{Tho01}) in the sense that it defines the gradient flow of the Kempf-Ness functional $\cJ$, which is globally convex on $\cH$. This GIT framework gives supporting evidence for the equation \eqref{TLPF} working well (\cf Remark \ref{GITint}).
\item From a PDE point of view, the TLPF is more in line with the parabolic equation proposed by Krylov \cite{Kry76} by Theorem \ref{tanconv}.
\end{itemize}

Before giving the second main theorem, we recall briefly some existence results of dHYM metrics. In the analysis of the equation \eqref{dHYMf}, it is crucial to give a proper notion of $C$-subsolutions whose existence implies a priori estimates to all orders. In this direction, Collins-Jacob-Yau \cite{CJY15} showed the following result by using the method of continuity;
\begin{thm}[\cite{CJY15}] \label{supere}
Let $X$ be a compact complex manifold with a K\"ahler form $\a$, and $\hat{\chi}$ a closed real $(1,1)$-form. Assume that $\hat{\Theta} \in ((n-2)\hpi, n \hpi)$ and there is a $C$-subsolution $\sphi$ satisfying $\Theta(A[\sphi])>(n-2)\hpi$. Then there exists a deformed Hermitian Yang-Mills metric \eqref{dHYMf}.
\end{thm}
We remark that for any $C$-subsolution $\sphi$, the supercritical phase condition $\Theta(A[\sphi])>(n-2)\hpi$ is automatically satisfied if $\hat{\Theta} \geq ((n-2)\hpi+\frac{2}{n}) \hpi$, and hence is not vacuous. However it is expected that the condition $\Theta(A[\sphi])>(n-2)\hpi$ can be improved when $\hat{\Theta} \in ((n-2)\hpi, ((n-2)\hpi+\frac{2}{n}))$. Pingali \cite{Pin19} showed that this is actually true when $n=3$ using a new continuity path obtained by rewriting \eqref{dHYMf} as a generalized complex Monge-Amp\`ere equation. In \cite{CJY15}, they also claimed that one can show a weaker existence result by using the LBMCF;
\begin{thm}[\cite{CJY15}, Remark 7.4] \label{hyperf}
Let $X$ be a compact complex manifold with a K\"ahler form $\a$, and $\hat{\chi}$ a closed real $(1,1)$-form. Assume that $\hat{\Theta} \in ((n-1)\hpi,n \hpi)$ and there is a $C$-subsolution $\sphi$ satisfying $\Theta(A[\sphi])>(n-1)\hpi$. Then the line bundle mean curvature flow \eqref{LBMCF} with $\phi_0:=\sphi$ exists for all time, and converges to the deformed Hermitian Yang-Mills metric in the $C^\infty$-topology.
\end{thm}
When proving Theorem \ref{hyperf}, the point is that the hypercritical phase condition $\Theta(A[\phi_0])>(n-1)\hpi$ is preserved under the LBMCF so that the operator $\Theta(A[\phi_t])$ remains to be concave. The author does not know whether the hypercritical phase condition $\Theta(A[\phi_0])>(n-1)\hpi$ in Theorem \ref{hyperf} can be replaced by $\phi_0 \in \cH$. One reason is that the LBMCF defines the gradient flow of the volume functional $\cV$ (\cf Section \ref{ntiform}), but we do not know whether $\cV$ is globally convex. In the original LMCF case it is conjectured by Thomas-Yau \cite[Section 7]{TY02} that we have to assume an additional condition, so called the ``flow-stability'' for the initial Lagrangian to obtain the convergence of the flow (see also \cite{Nev13} for detailed expositions).

Now we give the second main theorem. We show that the TLPF potentially has more global existence and convergence properties. With the aid of Theorem \ref{tanconv} and already known methods, it is standard to show;
\begin{thm} \label{convf}
Let $X$ be a compact complex manifold with a K\"ahler form $\a$, and $\hat{\chi}$ a closed real $(1,1)$-form. Assume that $\hat{\Theta} \in ((n-1)\hpi,n \hpi)$. For any $\phi_0 \in \cH$, let $\phi_t$ be the tangent Lagrangian phase flow \eqref{TLPF} starting from $\phi_0$. Then
\begin{enumerate}
\item The flow $\phi_t$ exists for all positive time.
\item Moreover, if there is a $C$-subsolution, then the flow $\phi_t$ converges to the deformed Hermitian Yang-Mills metric $\phi_\infty \in \cH$ in the $C^\infty$-topology.
\end{enumerate}
\end{thm}
Our notion of $C$-subsolutions coincides with that defined in \cite{CJY15}. Indeed, the analysis of the TLPF proceeds very closely to general theory of fully non-linear parabolic equations whose RHS is concave \cite{PT17} (based on the elliptic case \cite{Sze18}). We would like to make several comments on Theorem \ref{convf};
\begin{itemize}
\item In Theorem \ref{convf} the flow-stability assumption as  in \cite{TY02} is not needed. Moreover, the initial condition $\phi_0 \in \cH$ is sharp (otherwise, the TLPF \eqref{TLPF} is not well-defined). One should compare this with Neves's result \cite{Nev13} in which he showed that without the almost calibrated assumption a finite-time singularity occurs under the LMCF. We expect that the same is true of the LBMCF.
\item Theorem \ref{convf} removes the hypercritical phase condition $\Theta(A[\phi_0])>(n-1)\hpi$ imposed in Theorem \ref{hyperf}, and gives a new proof for the existence of dHYM metrics (we note that if $\hat{\Theta}>(n-1)\hpi$, then any $C$-subsolution $\sphi$ is almost calibrated by Remark \ref{subac}, and hence we can take this $\sphi$ as the initial data for instance).
\item The assumption of Theorem \ref{convf} is stronger than that of Theorem \ref{supere}. However we emphasize that Theorem \ref{convf} has a more natural meaning from geometric/variational point of view, and the estimates involved are much simpler than those in \cite{CJY15}.
\end{itemize}
Even if there are no canonical metrics, gradient flows (in particular, Theorem \ref{convf} (1)) are thought to be effective for constructing an optimally destabilizing one-parameter subgroups in GIT, or an analogue of the Harder-Narasimhan filtrations of unstable vector bundles. Now this problem catches a great deal of attention (for instance, see related results \cite{CHT17,DS16,His19,Sjo19,Xia19} studied in K\"ahler settings).

This paper is organized as follows. In Section \ref{fnd}, we fix some notations by following \cite{CXY17,JY17}, and study the basic properties of the space $\cH$ and functionals on it. In Section \ref{tcv}, we give a proof of Theorem \ref{tanconv} that is the core of this paper. In Section \ref{tlpf}, we prove some monotonicity formulas along the TLPF. Then we prove the long time existence of the flow which is the first part of Theorem \ref{convf}. In Section \ref{convfpC}, we recall the notion of $C$-subsolutions defined in \cite{CJY15}. Then in accordance with \cite{PT17}, we estabilish the $C^k$ estimates and the $C^\infty$-convergence of the TLPF, which shows the second part of Theorem \ref{convf}.
\begin{ackn}
The author expresses his gratitude to Prof. F.~R.~Harvey and H.~B.~Lawson, Jr. for pointing out a relation between Theorem \ref{tanconv} and their work \cite{HL20}.
\end{ackn}
\section{Foundations} \label{fnd}
\subsection{Notations and formulas} \label{ntiform}
Let $X$ be an $n$-dimensional compact complex manifold with a K\"ahler form $\a$ and $\hat{\chi}$ a closed real $(1,1)$-form. First we fix some notations. For $\phi \in C^\infty(X;\R)$, we define a Hermitian metric $\eta$\footnote{The metric $\eta$ is only used in the computation of $\ddt \cV$ in Proposition \ref{eff}.} on $T^{1,0}X$ by
\[
\eta_{i \bar{j}}:=\a_{i \bar{j}}+\chi_{i \bar{\ell}} \a^{k \bar{\ell}} \chi_{k \bar{j}}.
\]
Let
\[
v:=\bigg| \frac{(\a+\sqrt{-1}\chi_\phi)^n}{\a^n} \bigg|=\sqrt{\prod_{i=1}^n (1+\lambda_i^2)},
\]
where $\lambda_i$ denotes the eigenvalues of $A[\phi]$. In particular, this implies $v \geq 1$. In terms of $\cV$ and $\Theta(A[\phi])$, the $\C$-valued function $(\a+\sqrt{-1}\chi_\phi)^n/\a^n$ is expressed as
\[
\frac{(\a+\sqrt{-1}\chi_\phi)^n}{\a^n}=v e^{\sqrt{-1}\Theta(A[\phi])}.
\]
Integrating with respect to $\a^n$ we get
\[
e^{-\sqrt{-1}\hat{\Theta}}Z=\int_X v e^{\sqrt{-1}(\Theta(A[\phi])-\hat{\Theta})} \a^n=\int_X v\cos(\Theta(A[\phi])-\hat{\Theta})\a^n+\sqrt{-1} \int_X v \sin(\Theta(A[\phi])-\hat{\Theta}) \a^n.
\]
So by the definition of $\hat{\Theta}$ we know that $e^{-\sqrt{-1}\hat{\Theta}}Z$ is real so that
\[
|Z|=\int_X v\cos(\Theta(A[\phi])-\hat{\Theta})\a^n, \quad \int_X v \sin(\Theta(A[\phi])-\hat{\Theta}) \a^n=0.
\]
According to \cite[Section 2]{CY18}, we define functionals on the space of almost calibrated potentials $\cH$. We define the {\it Calabi-Yau functional} $CY_\C$ by
\[
CY_\C(\phi):=\frac{1}{n+1} \sum_{j=0}^n \int_X \phi (\a+\sqrt{-1} \chi_\phi)^j \wedge (\a+\sqrt{-1} \hat{\chi})^{n-j}, \quad \phi \in \cH.
\]
This is a $\C$-valued functional. Then the variational formula of $CY_\C$ is given by
\[
\d CY_\C(\d \phi)=\int_X \d \phi (\a+\sqrt{-1} \chi_\phi)^n.
\]
Also we set
\[
\cC(\phi):=\Re \big( e^{-\sqrt{-1}\hat{\Theta}} CY_\C(\phi) \big),
\]
\[
\cJ(\phi):=-\Im \big( e^{-\sqrt{-1}\hat{\Theta}} CY_\C(\phi) \big).
\]
Then the variational formula of $CY_\C$ yields that
\[
\d \cC(\d \phi)=\int_X \d \phi \Re \big( e^{-\sqrt{-1}\hat{\Theta}}(\a+\sqrt{-1}\chi_\phi)^n \big),
\]
\[
\d \cJ(\d \phi)=-\int_X \d \phi \Im \big( e^{-\sqrt{-1}\hat{\Theta}}(\a+\sqrt{-1}\chi_\phi)^n \big).
\]
Also we define the {\it volume functional} by
\[
\cV(\phi):=\int_X v_\phi \a^n.
\]
The function $\cV$ is non-negative. More precisely, we have $\cV(\phi) \geq |Z|$ for all $\phi \in \cH$ (\cf \cite[Proposition 3.2]{JY17}). The variational formula of $\cV$ (\cf \cite[Proposition 3.4]{JY17}) is given by
\[
\d \cV(\d \phi)=\int_X \langle d \Theta(A[\phi]), d \d \phi \rangle_\eta v_\phi \a^n.
\]
These functionals have the following properties;
\begin{prop}
For any $\phi \in \cH$ and $c \in \R$ we have
\begin{enumerate}
\item $CY_\C(\phi+c)=CY_\C(\phi)+cZ$.
\item $\cC(\phi+c)=\cC(\phi)+c|Z|$.
\item $\cJ(\phi+c)=\cJ(\phi)$.
\end{enumerate}
\end{prop}
We collect the properties of $f$ that we will need;
\begin{prop} \label{bps}
Suppose we have real numbers $\lambda_1 \geq \ldots \geq \ldots \lambda_n$ which satisfy $\theta(\lambda)=\s$ for $\s \in ((n-2)\hpi,n \hpi)$. Then $\lambda=(\lambda_1,\ldots,\lambda_n)$ have the following properties;
\begin{enumerate}
\item $\lambda_1 \geq \lambda_2 \geq \ldots \geq \lambda_{n-1}>0$ and $\lambda_{n-1} \geq |\lambda_n|$. In particular, $\lambda$ lies in the set $\{\lambda \in \R^n|\sum_{i=1}^n \lambda_i > 0 \}$.
\item The set $\{\lambda \in \R^n|\theta(\lambda) \geq \s\}$ is convex with boundary a smooth, convex hypersurface.
\end{enumerate}
Furthermore, if $\s \geq (n-1)\hpi$, then
\begin{enumerate}
\setcounter{enumi}{2}
\item $\lambda_n>0$.
\end{enumerate}
In addition, if $\s \geq (n-2)\hpi+\b$, then there exist constants $\e(\b)>0$ and $C(\b)>0$ such that
\begin{enumerate}
\setcounter{enumi}{3}
\item if $\lambda_n \leq 0$, then $\lambda_{n-1} \geq \e(\b)$.
\item $|\lambda_n|<C(\b)$.
\end{enumerate}
\end{prop}
\begin{proof}
The property (3) is trivial. See \cite[Lemma 3.1]{CJY15} for other statements.
\end{proof}
We end up this subsection with showing some properties of $F$. Let us write $F^{ij}$ for the derivative of $F$ with respect to the $ij$-entry of $A$. Then at a diagonal matrix $A$ we have
\begin{equation} \label{propfF}
F^{ij}=\d_{ij}f_i,
\end{equation}
\begin{equation} \label{propsF}
F^{ij,rs}=f_{ir}\d_{ij}\d_{rs}+\frac{f_i-f_j}{\lambda_i-\lambda_j}(1-\d_{ij})\d_{is}\d_{jr}.
\end{equation}
The first formula shows that the operator $F(A[\phi])$ is elliptic. In the second formula, we note that $\frac{f_i-f_j}{\lambda_i-\lambda_j} \leq 0$. In particular, we have $f_i \leq f_j$ if $\lambda_i \geq \lambda_j$ (see \cite[Section 4]{Sze18} for more details).
\subsection{Fundamental estimates on $\cH$}
First, by Proposition \ref{bps} (1), we have;
\begin{lem}[Green function estimate] \label{gfs}
For any $\phi \in \cH$, we have a uniform bound $\Delta_{\a} \phi \geq -C$ for some uniform constant $C>0$ depending only on $\a$ and $\hat{\chi}$. In particular, we have
\[
\sup_X \phi \leq \int_X \phi \a^n+C'
\]
for some uniform constant $C'>0$ depending only on $\a$ and $\hat{\chi}$.
\end{lem}
\begin{prop}[Harnack type inequality] \label{har}
Assume $\hat{\Theta} \in ((n-1)\hpi,n \hpi)$. Then for any $\phi \in \cH$, there exists a constant $C$ and $C'$ depending only on $\a$, $\hat{\chi}$, $\hat{\Theta}$, $\inf_X \Theta(A[\phi])$ and $\cC(\phi)$ such that
\[
\sup_X \phi \leq -C \inf_X \phi+C'.
\]
\end{prop}
\begin{proof}
For any $\phi \in \cH$, we set $\tilde{\phi}:=\phi-|Z|^{-1}\cC(\phi)$ so that $\cC(\tilde{\phi})=0$. We may assume that $0 \in \cH$. We connect $\tilde{\phi}$ with the base point $0$ by a segment $s\tilde{\phi}$ ($s \in [0,1]$), so we have $\chi_{s\tilde{\phi}}=s \chi_{\tilde{\phi}}+(1-s)\hat{\chi}$. Since $\hat{\Theta}>(n-1)\hpi$, we have a trivial upper bound
\[
\Theta(A[s \tilde{\phi}])-\hat{\Theta}<n \hpi-\hat{\Theta}<\hpi.
\]
Moreover, by \cite[Lemma 3.1 (7)]{CY18}, for all $s \in [0,1]$ we have
\begin{eqnarray*}
\Theta(A[s\tilde{\phi}])-\hat{\Theta} &\geq& \min \{ \Theta(A[0]), \Theta(A[\tilde{\phi}]) \}-\hat{\Theta}\\
&\geq& \min \{ \inf_X \Theta(A[0]), \inf_X \Theta(A[\phi]) \}-\hat{\Theta}\\
&>&-\hpi
\end{eqnarray*}
(indeed, the set $\cH$ is convex when $\hat{\Theta}>(n-1)\hpi$ as pointed out in \cite[Section 2]{CY18}). Combining with $v_{s \tilde{\phi}} \geq 1$ we have
\[
\Re \big( e^{-\sqrt{-1}\hat{\Theta}}(\a+\sqrt{-1}\chi_{s \tilde{\phi}})^n \big)=v_{s \tilde{\phi}} \cos (\Theta(A[s \tilde{\phi}])-\hat{\Theta}) \a^n \geq c \a^n
\]
for some constant $c>0$ depending only on $\hat{\chi}$, $\hat{\Theta}$ and $\inf_X \Theta(A[\phi])$. Thus $\nu_{\tilde{\phi}}:=\int_0^1\Re \big( e^{-\sqrt{-1}\hat{\Theta}}(\a+\sqrt{-1}\chi_{s\tilde{\phi}})^n \big) ds$ defines a positive measure with volume $|Z|$ and a uniform lower bound $c\a^n$. On the other hand, the variational formula of $\cC$ implies that
\[
0=\cC(\tilde{\phi})=\int_0^1 \int_X \tilde{\phi} \Re \big( e^{-\sqrt{-1}\hat{\Theta}}(\a+\sqrt{-1}\chi_{s\tilde{\phi}})^n \big) ds=\int_X \tilde{\phi} \nu_{\tilde{\phi}}.
\]
In particular, this shows that $\sup_X \tilde{\phi} \geq 0$ and $\inf_X \tilde{\phi} \leq 0$. Keeping this in mind, we compute
\[
\int_X \tilde{\phi} \a^n=-\frac{1}{c} \int_X \tilde{\phi}(\nu_{\tilde{\phi}}-c\a^n) \leq - \frac{1}{c} \bigg[ \int_X (\nu_{\tilde{\phi}}-c\a^n) \bigg] \inf_X \tilde{\phi}=-C_1 \inf_X \tilde{\phi},
\]
where we put $C_1:=c^{-1}|Z|-[\a]^n$. Combining with the Green function estimate (\cf Lemma \ref{gfs}) we have
\[
\sup_X \tilde{\phi} \leq \int_X \tilde{\phi} \a^n+C_2 \leq -C_1 \inf_X \tilde{\phi}+C_2.
\]
Putting $\tilde{\phi}=\phi-|Z|^{-1}\cC(\phi)$ into the above, we obtain the desired formula.
\end{proof}
\section{Tan-concavity} \label{tcv}
Now we give a proof of Theorem \ref{tanconv}.
\begin{proof}[Proof of Theorem \ref{tanconv}]
From a direct computation, we have
\[
df=(1+f^2)d \theta, \quad d \theta=\sum_i \frac{d \lambda_i}{1+\lambda_i^2},
\]
\[
\nabla^2 \theta=-\sum_i \frac{2 \lambda_i}{(1+\lambda_i)^2} d\lambda_i \otimes d\lambda_i,
\]
\begin{eqnarray*}
\nabla^2 f &=& 2f df \otimes d \theta+(1+f^2) \nabla^2 \theta \\
&=& 2f(1+f^2) d \theta \otimes d \theta+(1+f^2) \nabla^2 \theta \\
&=& (1+f^2)(2f d \theta \otimes d \theta+\nabla^2 \theta).
\end{eqnarray*}
If we set $T(\lambda):=\tan((n-1)\hpi-\theta(\lambda))$, we observe that $f<-T$ since
\[
-\hpi<(n-1)\hpi-\theta < \hat{\Theta}-\theta<\hpi
\]
from the assumption. Thus
\[
\nabla^2 f \leq (1+f^2)(-2T d \theta \otimes d \theta+\nabla^2 \theta)
\]
since $d \theta \otimes d \theta$ is semipositive. In the standard coordinates, the form $-\frac{1}{2}(-2T d \theta \otimes d \theta+\nabla^2 \theta)$ has the following matrix representation
\begin{eqnarray*}
M:=\bigg( \frac{T+\d_{ij} \lambda_i}{(1+\lambda_i^2)(1+\lambda_j^2)} \bigg).
\end{eqnarray*}
Then it suffices to show that $M \geq 0$, \ie $M$ is positive semidefinite. For $k=1,\ldots,n$ and $I \subset \{1,\ldots,n \}$, let $M_I$ be the principal submatrix of $M$ associated to $I$, \ie
\[
M_I=\bigg( \frac{T+\d_{ij} \lambda_i}{(1+\lambda_i^2)(1+\lambda_j^2)} \bigg)_{i,j \in I}.
\]
Since $M$ is symmetric, $M \geq 0$ holds if and only if $\det M_I \geq 0$ for all $I \subset \{1,\ldots,n\}$. Moreover, since the term $(1+\lambda_i^2)^{-1}$ appears precisely in the $(i,k)$-entries or $(k,i)$-entries of $M$ for $k=1,\ldots,n$, we see that
\[
\prod_{i \in I} (1+\lambda_i^2)^2 \cdot \det M_I=\det \tilde{M}_I,
\]
where $\tilde{M}_I$ denotes the principal submatrix of the $n \times n$ symmetric matrix $\tilde{M}:=(T+\d_{ij}\lambda_i)$. By the symmetry of $f$, we may only consider the point $\lambda \in \cS$ with $\lambda_1 \geq \ldots \geq \lambda_n$. A standard induction argument shows that
\[
\det \tilde{M}_I=\prod_{i \in I} \lambda_i \cdot \bigg(1+\sum_{i \in I} \frac{T}{\lambda_i} \bigg)
\]
as long as $\lambda_n \neq 0$. Set $x_i:=\arctan \lambda_i$, $y_i:=\hpi-x_i$ ($i=1,\ldots,n$). The argument is divided into two cases;

\vspace{2mm}

\noindent
{\bf Case 1 ($\theta(\lambda) \geq (n-1)\hpi$)}; In this case, we know that $T \leq 0$ and $\lambda_i>0$ for all $i=1,\ldots,n$. So the condition $\det \tilde{M}_I \geq 0$ is equivalent to
\[
1+\sum_{i \in I} \frac{T}{\lambda_i} \geq 0.
\]
This is clearly true for all $I$ when $T=0$, so we assume $T<0$, or equivalently $\theta(\lambda)>(n-1)\hpi$. Since $\sum_{i \in I} \frac{T}{\lambda_i} \geq \sum_{i=1}^n \frac{T}{\lambda_i}$, we may only consider the case $I=\{1,\ldots,n \}$. From the assumption, we observe that
\[
0<\sum_{i=1}^n y_i<\hpi, \quad 0<y_1 \leq \ldots \leq y_n<\hpi.
\]
So combining with the formula
\[
\tan \big(\sum_{i=1}^n y_i \big)=\frac{\tan(\sum_{i=1}^{n-1} y_i)+\tan y_n}{1-\tan(\sum_{i=1}^{n-1} y_i)\tan y_n},
\]
we know that $0<1-\tan \big(\sum_{i=1}^{n-1} y_i \big)\tan y_n<1$, and hence
\[
\tan \big(\sum_{i=1}^n y_i \big) \geq \tan \big(\sum_{i=1}^{n-1} y_i \big)+\tan y_n.
\]
Repeating this, we obtain
\[
-\frac{1}{T}=\tan \big( \sum_{i=1}^n y_i \big) \geq \sum_{i=1}^n \tan y_i=\sum_{i=1}^n \frac{1}{\lambda_i}.
\]

\vspace{2mm}

\noindent
{\bf Case 2 ($\hat{\Theta}-\hpi<\theta(\lambda)<(n-1)\hpi$)}; In this case, we have $T>0$. If $\lambda_n \geq 0$, we know that $\tilde{M} \geq 0$ since it is decomposed as
\[
\tilde{M}=T\cdot E_n+{\rm diag}(\lambda_1,\ldots,\lambda_n),
\]
where $E_n$ denotes the $n \times n$ matrix whose all entries are equal to one. Now we assume that $\lambda_n<0$. We note that $\lambda_i>0$ for $i=1,\ldots,n-1$. If $n \notin I=\{i_1,\ldots, i_\ell \}$, we have a decomposition $\tilde{M}_I=T \cdot E_\ell+{\rm diag}(\lambda_{i_1},\ldots,\lambda_{i_\ell})$ with $\lambda_{i_1}, \ldots \lambda_{i_\ell}>0$, so $\det \tilde{M}_I \geq 0$ is clear. If $n \in I$, we have $\sum_{i \in I} \frac{T}{\lambda_i} \leq \sum_{i=1}^n \frac{T}{\lambda_i}$. Eventually we may assume $I=\{1,\ldots,n \}$. Then the condition $\det \tilde{M}_I \geq 0$ is equivalent to
\[
1+\sum_{i \in I} \frac{T}{\lambda_i} \leq 0.
\]
From the assumption, we have
\[
0<\sum_{i=1}^{n-1} y_i<\hpi<\sum_{i=1}^n y_i<\pi, \quad 0<y_1 \leq \ldots \leq y_{n-1}<\hpi<y_n<\pi.
\]
So from the formula
\[
\tan \big(\sum_{i=1}^n y_i \big)=\frac{\tan(\sum_{i=1}^{n-1} y_i)+\tan y_n}{1-\tan(\sum_{i=1}^{n-1} y_i)\tan y_n},
\]
we know that $1<1-\tan(\sum_{i=1}^{n-1} y_i)\tan y_n$, $\tan(\sum_{i=1}^{n-1} y_i)+\tan y_n<0$ and hence we obtain
\[
\tan \big(\sum_{i=1}^n y_i \big) \geq \tan \big( \sum_{i=1}^{n-1} y_i \big)+\tan y_n.
\]
Applying the same argument as in Case 1 to the first term, we obtain $\tan(\sum_{i=1}^{n-1} y_i) \geq  \sum_{i=1}^{n-1} \tan y_i$. Thus
\[
-\frac{1}{T}=\tan \big( \sum_{i=1}^n y_i \big) \geq \sum_{i=1}^n \tan y_i=\sum_{i=1}^n \frac{1}{\lambda_i}.
\]
This completes the proof.
\end{proof}
\begin{rk}
When $\hat{\Theta}<(n-1)\hpi$ the function $f$ is no longer concave since its level set is no longer convex (\cf \cite{Yua05}).
\end{rk}
\section{The tangent Lagrangian phase flow} \label{tlpf}
\subsection{Monotonicity formulas}
We start with some monotonicity properties of functionals along the TLPF defined in Section \ref{fnd}.
\begin{prop} \label{eff}
Along the TLPF $\phi_t$ with $\phi_0 \in \cH$ we have
\begin{enumerate}
\item $\cC(\phi_t)$ is constant.
\item $\cJ(\phi_t)$ is monotonically decreasing.
\item $\cV(\phi_t)$ is monotonically decreasing.
\end{enumerate}
\end{prop}
\begin{proof}
From the variational formula of $\cC$ and $\cJ$, we compute
\begin{eqnarray*}
\ddt \cC(\phi_t) &=& \int_X \ddt \phi \cdot \Re \big( e^{-\sqrt{-1}\hat{\Theta}}(\a+\sqrt{-1}\chi_\phi)^n \big)\\
&=& \int_X \tan(\Theta(A[\phi])-\hat{\Theta}) \Re \big( e^{-\sqrt{-1}\hat{\Theta}}(\a+\sqrt{-1}\chi_\phi)^n \big)\\
&=& \int_X v \sin (\Theta(A[\phi])-\hat{\Theta}) \a^n \\
&=& 0,
\end{eqnarray*}
\begin{eqnarray*}
\ddt \cJ(\phi_t) &=& -\int_X \ddt \phi \cdot \Im \big( e^{-\sqrt{-1}\hat{\Theta}}(\a+\sqrt{-1}\chi_\phi)^n \big)\\
&=& -\int_X \tan(\Theta(A[\phi])-\hat{\Theta}) \Im \big( e^{-\sqrt{-1}\hat{\Theta}}(\a+\sqrt{-1}\chi_\phi)^n \big)\\
&=& -\int_X \tan^2(\Theta(A[\phi])-\hat{\Theta}) \Re \big( e^{-\sqrt{-1}\hat{\Theta}}(\a+\sqrt{-1}\chi_\phi)^n \big)\\
&\leq& 0
\end{eqnarray*}
since $\phi_t$ stays in the set $\cH$ as long as it exists (\cf Lemma \ref{sff}) and hence the form $\Re \big( e^{-\sqrt{-1}\hat{\Theta}}(\a+\sqrt{-1}\chi_{\phi_t})^n \big)$ defines a positive measure on $X$. Also we have
\[
\ddt \cV(\phi_t) =-\int_X \langle d \Theta (A[\phi]), d \ddt \phi \rangle_\eta v_\phi \a^n=-\int_X (1+F^2)|d \Theta (A[\phi])|_\eta^2 v_\phi \a^n \leq 0.
\]
\end{proof}
\begin{rk} \label{GITint}
In \cite[Section 2]{CY18}, they discovered a GIT/moment map interpretation for dHYM metrics in which the $\cJ$-functional plays a role of the Kempf-Ness functional, and dHYM metrics are characterized as critical points of $\cJ$. Also the space $\cH$ has a natural Riemannian structure defined by
\[
\|\d \phi\|_{\phi}^2:=\int_X (\d \phi)^2 \Re \big( e^{-\sqrt{-1}\hat{\Theta}}(\a+\sqrt{-1}\chi_\phi)^n \big), \quad \d \phi \in T_\phi \cH.
\]
Moreover, the Riemannian manifold $\cH$ is equipped with the Levi-Civita connection, and the sectional curvature is non-positive as shown in the recent work \cite{CCL20}. The $\cJ$-functional is convex along geodesics with respect to this Riemannian structure. From the proof of Proposition \ref{eff}, we know that the TLPF defines the gradient flow of the $\cJ$-functional.
\end{rk}
By Lemma \ref{sff}, Proposition \ref{eff} (1) and Proposition \ref{har} we obtain;
\begin{cor} \label{harff}
Assume $\hat{\Theta} \in ((n-1)\hpi,n \hpi)$. Then along the TLPF $\phi_t$, the Harnack type inequality
\[
\sup_X \phi_t \leq -C \inf_X \phi_t+C'
\]
holds for some uniform constant $C, C'>0$ depending only on $\a$, $\hat{\chi}$, $\hat{\Theta}$ and the initial data $\phi_0 \in \cH$.
\end{cor}
\begin{rk} \label{cff}
In the proof of Proposition \ref{har}, we see that the normalization $\tilde{\phi}=\phi-|Z|^{-1}\cC(\phi)$ is a significant issue. As for the LBMCF \eqref{LBMCF}, the $\cC(\phi_t)$ is not constant or even monotone. Instead, one can easily show that if $\xi<\inf_X \hat{\Theta}(A[\phi_0]) \leq \sup_X \Theta(A[\phi_0])<\xi+\hpi$ for some $\xi \in \R$, then the $\xi$-twisted $\cC$-functional
\[
\cC_\xi(\phi):=\Re \big( e^{-\sqrt{-1}\xi} CY_\C(\phi) \big)
\]
is decreasing along the flow by applying Jensen's inequality to $\tan(x-\xi)$ for $x \in (\xi,\xi+\hpi)$ (see \cite[Proposition 2.1]{Tak19}). However this argument does not apply when $\osc_X \Theta(A[\phi_0]) \geq \hpi$.
\end{rk}
\subsection{Long time existence}
Now let us consider the TLPF $\phi_t$ with $\phi_0 \in \cH$ for $t \in [0,T)$ (where $T>0$ is not necessarily the maximal existence time).
\begin{lem} \label{sff}
Along the TLPF $\phi_t$ with $\phi_0 \in \cH$, we have a uniform control
\[
\inf_X \Theta(A[\phi_0]) \leq \Theta(A[\phi_t]) \leq \sup_X \Theta(A[\phi_0]).
\]
So $\|\ddt \phi \|_{C^0} \leq C$ for some constant $C>0$ depending only on $\hat{\Theta}$ and $\phi_0$. In particular, the flow stays in $\cH$ as long as it exists.
\end{lem}
\begin{proof}
A straightforward computation shows that
\[
\ddt F(A[\phi_t])=F^{i\bar{j}} \p_i \p_{\bar{j}} \ddt \phi=F^{i\bar{j}} \p_i \p_{\bar{j}}(F(A[\phi_t])).
\]
So by a maximum principle we have $\inf_X F(A[\phi_0]) \leq F(A[\phi_t]) \leq \sup_X F(A[\phi_0])$, which shows that $\phi_t \in \cH$ as long as it exists. So by the monotonicity of $\tan \colon (-\hpi,\hpi) \to \R$, we have $\inf_X \Theta(A[\phi_0]) \leq \Theta(A[\phi_t]) \leq \sup_X \Theta(A[\phi_0])$ as desired.
\end{proof}
Integrating $\ddt \phi$ on $[0,T)$ we obtain;
\begin{cor}
Along the TLPF $\phi_t$ with $\phi_0 \in \cH$, we have $\|\phi_t\|_{C^0} \leq C_T$ for some constant $C_T>0$ depending only on $\hat{\Theta}$, $\phi_0$ and $T$.
\end{cor}
Also, combining with Proposition \ref{bps} we have;
\begin{cor} \label{unifev}
Along the TLPF $\phi_t$ with $\phi_0 \in \cH$, there is a uniform constant $C>0$ depending only on $\phi_0$ such that $|\lambda_n|<C$.
\end{cor}
Set
\[
\cF(\lambda[\phi]):=\sum_{i=1}^n f_i(\lambda[\phi]).
\]
Then Lemma \ref{sff} further implies;
\begin{cor} \label{lbffc}
For the TLPF $\phi_t$ with $\phi_0 \in \cH$, there is a constant $C>0$ depending only on $\hat{\Theta}$ and $\phi_0$ such that
\begin{equation}
\frac{1}{C}<\cF(\lambda[\phi_t])<C.
\end{equation}
\end{cor}
\begin{proof}
We compute
\[
\cF(\lambda[\phi_t])=(1+f^2) \sum_{i=1}^n \frac{1}{1+\lambda_i^2}.
\]
The upper bound of $f$ follows from the assumption $\hat{\Theta}>(n-1)\hpi$. The lower bound of $f$ is uniformly controlled by $\inf_X \Theta(A[\phi_0])$ by Lemma \ref{sff}. Combining with the uniform control of $|\lambda_n|$ we obtain the desired estimate.
\end{proof}
Set
\[
\cL:=\ddt-F^{k\bar{k}} \nabla_k \nabla_{\bar{k}}.
\]
\begin{lem} \label{secdl}
Assume $\hat{\Theta} \in ((n-1)\hpi,n \hpi)$ and let $\phi_t$ ($t \in [0,T)$) be the TLPF with $\phi_0 \in \cH$. Then we have
\[
|\dd \phi_t|_\a \leq C_T,
\]
where the constant $C_T>0$ depends only on $\a$, $\hat{\chi}$, $\hat{\Theta}$, $\phi_0$ and $T$.
\end{lem}
\begin{proof}
Take $T'<T$ and let $\nabla$ be the Chern connection with respect to $\a$. The strategy is applying the maximum principle to the function
\[
G:=\log \lambda_1-Dt
\]
on $X \times [0,T']$. The constant $D>0$ is determined in later argument. Assume that the function $G$ attains its maximum on $X \times [0,T']$ at some $(x_0,t_0)$. We want to apply the operator $\cL$ to $G$. More precisely, since $\lambda_1$ may not be differentiable ($\lambda_1$ is only continuous on $X$), we use the perturbation technique as in \cite[Section 4]{Sze18}. We take a normal coordinates with respect to $\a$ centered at $x_0$ to identify $A[\phi(x_0,t_0)]$ as a matrix valued function on it which is diagonal at the origin with eigenvalues $\lambda_1 \geq \ldots \geq \lambda_n$, and then adjust $A$ by subtracting a small constant diagonal matrix $B={\rm diag}(B^{ii})$ with $0=B^{11}<B^{22}<\ldots<B^{nn}$. At the origin the matrix $\tilde{A}:=A-B$ has eigenvalues
\[
\tlambda_1=\lambda_1, \quad \tlambda_i=\lambda_i-B^{ii}<\tlambda_1 \;\; \text{($i>1$)}.
\]
These are distinct, and define smooth functions near the origin. Set
\[
\tilde{G}:=\log \tlambda_1-Dt.
\]
Then we have $\tilde{G}(x,t) \leq G(x,t)$ and $\tilde{G}$ achieves its maximum $\tilde{G}(x_0,t_0)=G(x_0,t_0)$ at $(x_0,t_0)$. It suffices to show that $\tlambda_1$ is bounded from above. We may assume $\lambda_1 \geq 1$ at the origin. We compute
\[
\cL \log \tlambda_1=\frac{1}{\lambda_1} \cL \tlambda_1+F^{k\bar{k}} \frac{|\nabla_{\bar{k}} \tlambda_1|^2}{\lambda_1^2},
\]
\[
\nabla_{\bar{k}} \tlambda_1=\nabla_{\bar{k}} \chi_{1\bar{1}}-\nabla_{\bar{k}} B^{11},
\]
\begin{eqnarray*}
\nabla_k \nabla_{\bar{k}} \tlambda_1&=&\nabla_k \nabla_{\bar{k}} \chi_{1\bar{1}}+\sum_{p>1} \frac{|\nabla_{\bar{k}} \chi_{1 \bar{p}}|^2+|\nabla_{\bar{k}}\chi_{p \bar{1}}|^2}{\lambda_1-\tlambda_p}\\
&+& \nabla_k \nabla_{\bar{k}} B^{11}-2 \Re \sum_{p>1} \frac{\nabla_k \chi_{p \bar{1}} \nabla_{\bar{k}} B^{1\bar{p}}+\nabla_k \chi_{1\bar{p}} \nabla_{\bar{k}} B^{p \bar{1}}}{\lambda_1-\tlambda_p}+\tlambda_1^{pq,rs} \nabla_k B^{pq} \nabla_{\bar{k}} B^{rs}
\end{eqnarray*}
(for instance, see \cite[equation (70)]{Sze18}). Evaluating this expression at the origin, and using that $B$ is constant we have
\[
\nabla_{\bar{k}} \tlambda_1=\nabla_{\bar{k}} \chi_{1\bar{1}},
\]
\[
\nabla_k \nabla_{\bar{k}} \tlambda_1= \nabla_k \nabla_{\bar{k}} \chi_{1\bar{1}}+\sum_{p>1} \frac{|\nabla_{\bar{k}} \chi_{1 \bar{p}}|^2+|\nabla_{\bar{k}}\chi_{p \bar{1}}|^2}{\lambda_1-\tlambda_p}.
\]
On the other hand, the evolution equation of the TLPF implies that
\[
\ddt \tlambda_1=\ddt \nabla_1 \nabla_{\bar{1}} \phi=F^{l \bar{k},s\bar{r}} \nabla_{\bar{1}} \chi_{l \bar{k}} \nabla_1 \chi_{s \bar{r}}+F^{k\bar{k}} \nabla_1 \nabla_{\bar{1}} \chi_{k \bar{k}}.
\]
Thus we compute $\cL \tlambda_1$ as
\begin{eqnarray*}
\cL \tlambda_1 &=& F^{k\bar{k}} \bigg( \nabla_1 \nabla_{\bar{1}} \chi_{k \bar{k}}-\nabla_k \nabla_{\bar{k}} \chi_{1\bar{1}}-\sum_{p>1} \frac{|\nabla_{\bar{k}} \chi_{1 \bar{p}}|^2+|\nabla_{\bar{k}}\chi_{p \bar{1}}|^2}{\lambda_1-\tlambda_p} \bigg)\\
&+& F^{l \bar{k},s\bar{r}} \nabla_{\bar{1}} \chi_{l \bar{k}} \nabla_1 \chi_{s \bar{r}}.
\end{eqnarray*}
The first two terms are estimated as
\begin{eqnarray*}
\nabla_1 \nabla_{\bar{1}} \chi_{k \bar{k}}-\nabla_k \nabla_{\bar{k}} \chi_{1\bar{1}} &=& \nabla_1 \nabla_{\bar{1}} \hat{\chi}_{k \bar{k}}-\nabla_k \nabla_{\bar{k}} \hat{\chi}_{1 \bar{1}}+\nabla_1 \nabla_{\bar{1}} \phi_{k \bar{k}}-\nabla_k \nabla_{\bar{k}} \phi_{1 \bar{1}}\\
& \leq & C_1+\Rm \ast \nabla \bnabla \phi \\
&\leq& C_2(\lambda_1+1)
\end{eqnarray*}
since $\nabla \bnabla \phi$ is controlled by $\lambda_1$, where $\Rm$ denotes the Riemannian curvature tensor of $\a$, and the constants $C_1$, $C_2$ only depends on $\a$ and $\hat{\chi}$. Thus we obtain
\begin{equation} \label{elogl}
\cL \log \tlambda_1 \leq C_2(1+\lambda_1^{-1}) \cF+\frac{1}{\lambda_1} F^{l \bar{k},s\bar{r}} \nabla_{\bar{1}} \chi_{l \bar{k}} \nabla_1 \chi_{s \bar{r}}+F^{k\bar{k}} \frac{|\nabla_{\bar{k}} \tlambda_1|^2}{\lambda_1^2}.
\end{equation}
For the first term of \eqref{elogl}, we have $C_2(1+\lambda_1^{-1}) \cF<C_3$ at the origin by $\lambda_1 \geq 1$ and Corollary \ref{lbffc}. From the concavity of $f$, the second term is non-positive. The third term is zero at $(x_0,t_0)$ by $\nabla \tilde{G}=0$. Thus applying the maximum principle to the function $\tilde{G}$ with $D:=C_3+1$, we conclude that $t_0=0$. This gives the desired bound.
\end{proof}
With the $C^2$-estimate in hand, we obtain a uniform control of the eigenvalues along the flow. Moreover if we assume $\hat{\Theta} \in ((n-1)\hpi, n \hpi)$, then the operator $F(A[\phi_t])$ in the RHS of \eqref{TLPF} is uniformly elliptic and concave. So we apply the Evans-Krylov theory \cite{Kry82,Wan12} to obtain;
\begin{lem} \label{evkry}
Let $\phi_t$ be the TLPF with $\phi_0 \in \cH$. Assume $\hat{\Theta} \in ((n-1)\hpi, n \hpi)$ and $\|\dd \phi \|_{C^0} \leq C_0$. Then there exist constants $C>0$ and $\b \in (0,1)$ depending only on $\a$, $\hat{\chi}$ and $C_0$ such that
\[
\|\dd \phi \|_{C^\b(X \times [0,T))} \leq C.
\]
\end{lem}
The higher order regularity of the flow follows from the Schauder estimates and a standard bootstrapping argument. We omit the detailed proofs. Finally, a standard argument using Ascoli-Arzel\`a theorem shows that;
\begin{thm}
Let $\phi_t$ be the tangent Lagrangian phase flow with $\phi_0 \in \cH$. Assume that $\|\dd \phi \|_{C^\b(X \times [0,T))}$ is uniformly controlled for some $\b \in (0,1)$ (where the constant $\b$ may depend on $T$). Then the flow $\phi_t$ extends beyond $T$. In particular, if $\hat{\Theta} \in ((n-1)\hpi, n \hpi)$, the above assumption is automatically satisfied, and hence the flow $\phi_t$ exists for all positive time.
\end{thm}
\section{Convergence of the TLPF under the existence of a $C$-subsolution} \label{convfpC}
\subsection{$C$-subsolutions}
Let $\Gamma_n$ be the positive orthant of $\R^n$. Collins-Jacob-Yau \cite{CJY15} introduced the notion of $C$-subsolutions;
\begin{dfn} \label{ecs}
A function $\sphi \in C^\infty(X;\R)$ is called a $C$-subsolution if for any $x \in X$, the set
\[
\big\{ \mu \in \Gamma_n \big| \theta(\lambda[\sphi(x)]+\mu)=\hat{\Theta} \big\}
\]
is bounded.
\end{dfn}
\begin{rk} \label{subac}
From \cite[Lemma 3.3]{CJY15}, any $C$-subsolution $\sphi$ must satisfy $\Theta(A[\sphi])>\frac{n}{n-1}(\hat{\Theta}-\hpi)$. In particular, any $C$-subsolution $\sphi$ is almost calibrated  when $\hat{\Theta}>(n-1)\hpi$.
\end{rk}
In particular, a genuine solution to \eqref{dHYMf} is clearly a $C$-subsolution. In \cite{CJY15}, the notion of $C$-subsolutions is used to study the elliptic equation \eqref{dHYMf}. In the next subsection, we will see that the same notion is also useful to study the limiting behavior of the TLPF. Set
\[
g(\lambda,\tau):=f(\lambda)+\tau, \quad (\lambda,\tau) \in \cS \times \R.
\]
The condition $f_i>0$ shows that at each point $(x,t)$, the ray $\{(\lambda[\sphi(x,t)]+s\mu,s\tau)|s \geq 0\}$ generated by any non-zero element $(\mu,\tau) \in \overline{\Gamma}_n \times \R_{\geq 0}$ intersects transversely with the level set $\{g=0\}$ just once. So by the compactness, there is a $\d>0$ and $K>0$ such that at each $(x,t) \in X \times [0,T)$, any element in the set
\begin{equation} \label{dKsub}
\big\{ (\mu,\tau) \in \Gamma_n \times \R_{\geq 0} \big| f(\lambda[\sphi(x,t)]-\d I+\mu)+\tau-\d=0 \big\}
\end{equation}
satisfies $|\mu|+|\tau|<K$, where $I$ denotes the vector $(1,\ldots,1)$ of eigenvalues of the identity matrix. In later arguments, we fix this $\d$ and $K$.
\subsection{Up to $C^k$-estimates}
In the remaining of the paper, we prove the second part of Theorem \ref{convf}. Again we note that the proof is mostly based on general theory of fully non-linear parabolic equations \cite{PT17}. However our function $f$ does not have the structural properties imposed in \cite{PT17,Sze18}. On the contrary, the function $f$ can not be extended to a symmetric cone $\Gamma \subset \R^n$ containing $\cS$ since $f(\lambda) \to -\infty$ as $\lambda$ reaches the boundary $\p \cS$. For this reason, we need to check carefully to see if every argument in \cite{PT17} carries over to our case. We will explain there is no substantial differences from \cite{PT17} except the gradient estimate (\cf Lemma \ref{gradb}).

In what follows we assume that $\hat{\Theta} \in ((n-1)\hpi,n \hpi)$ and there is a $C$-subsolution $\sphi$. Set $\hat{\chi}=\chi_{\sphi}$ and let $\phi_t$ be the TLPF with $\phi_0 \in \cH$. One can prove the following two lemmas exactly as in \cite{PT17};
\begin{lem}[see \cite{PT17}, Lemma 1]
There exists a uniform constant $C>0$ depending only on $\a$, $\hat{\chi}$, $\hat{\Theta}$ and $\phi_0$ such that $\|\phi_t\|_{C^0} \leq C$.
\end{lem}
\begin{proof}
This lemma is based on the parabolic version of the Alexandrov-Bakelman-Pucci estimates due to \cite{Tso85}. It is straightforward to check that the proof requires only the lower bound $\Delta_\a \phi_t \geq -C$ (\cf Lemma \ref{gfs}), the ellipticity of the operator and the boundedness of the set \eqref{dKsub}. Unlike the elliptic case, the argument for parabolic case is more subtle, which just provides a uniform lower bound for $\phi_t$ as mentioned in \cite{PT17}. So we apply the Harnack type equality (\cf Corollary \ref{harff}) to get the full estimate of $\phi_t$.
\end{proof}
\begin{lem}[see \cite{PT17}, Lemma 3] \label{key}
There exists a constant $\r=\r(\d,K)>0$ (where the constants $\d$, $K$ are defined in \eqref{dKsub}) so that if $|\lambda[\phi_t]-\lambda[\sphi]|>K$, then either
\[
\cL \phi_t>\r \cF(\lambda[\phi_t])
\]
or we have for any $i=1,\ldots,n$,
\[
F^{ii}(A[\phi_t])>\r \cF(\lambda[\phi_t]).
\]
\end{lem}
\begin{proof}
The proof requires only the ellipticity and convexity of the level set of $g$, that are available in our case.
\end{proof}
\begin{rk} \label{efc}
As for the elliptic operator $\Theta$, Collins-Jacob-Yau \cite[Proposition 3.5]{CJY15} proved a similar inequality based on \cite[Proposition 6]{Sze18} only by using the convexity of the level set of $\theta$. Indeed as pointed out in \cite{CJY15}, the proof of \cite[Proposition 6]{Sze18} only requires the ellipticity and the convexity of the level set of $\theta$. However this argument can not be extended directly to the parabolic case, \ie the LBMCF case \eqref{LBMCF}; if we set $h(\lambda,\tau):=\theta(\lambda)+\tau$, then the each level set $h(\lambda,\tau)=c$ defines the graph of the function $\tau=c-\theta(\lambda)$, which is convex if and only if the function $\theta$ itself is concave. This fails as soon as $\theta(\lambda)<(n-1)\hpi$.
\end{rk}
\begin{lem} \label{secdoest}
We have the estimate
\[
|\dd \phi_t|_\a \leq C(1+\sup_{X \times [0,T)}|\nabla \phi_t|_\a^2),
\]
where the constant $C>0$ depends only on $\a$, $\hat{\chi}$, $\hat{\Theta}$ and $\phi_0$. 
\end{lem}
The proof of this lemma also proceeds along the same line as in \cite[Lemma 2]{PT17}. However, using the uniform control of $|\lambda_n|$ and the vanishing of the torsion tensor of $\a$, we can simplify the argument. We give a proof for the sake of completeness.
\begin{proof}[Proof of Lemma \ref{secdoest}]
Take $T'<T$ and consider the function
\begin{equation} \label{functG}
G:=\log \lambda_1+\Phi(|\nabla \phi|^2)+\Psi(\phi)
\end{equation}
on $X \times [0,T']$, where the functions $\Phi$, $\Psi$ are specified by
\[
\Phi(s):=-\frac{1}{2} \log \bigg( 1-\frac{s}{2P} \bigg), \quad s \in [0,P],
\]
\[
\Psi(s):=D e^{-s}, \quad s \in \big[ \inf_{X \times [0,T']} \phi, \sup_{X \times [0,T']} \phi \big],
\]
where $P:=\sup_{X \times [0,T']}(|\nabla \phi|^2+1)$, and the large constant $D>0$ is chosen in the course of the proof. Then we note that
\[
\frac{1}{4P}<\Phi'<\frac{1}{2P}, \quad \Phi''=2(\Phi')^2>0.
\]
Assume that $G$ attains its maximum on $X \times [0,T']$ at some $(x_0,t_0)$. Now we use a perturbation technique similar to the one used in Lemma \ref{secdl}. We will apply the maximum principle to the function
\begin{equation} \label{functGt}
\tilde{G}:=\log \tlambda_1+\Phi(|\nabla \phi|^2)+\Psi(\phi),
\end{equation}
where we adopt the same notations as in the proof of Lemma \ref{secdl}. Since $|\lambda_n|$ is uniformly controlled along the flow, we may assume that at the origin, $\lambda$ satisfies;
\begin{itemize}
\item $\lambda_1 \geq 1$.
\item $|\lambda[\phi(x_0,t_0)]-\lambda[\sphi(x_0)]|>K$
\item $\kappa \lambda_1 \geq -\lambda_n$.
\item $\frac{1}{1+\lambda_1^2} \leq \frac{\r}{1+\lambda_n^2}$
\end{itemize}
where the constant $\r=\r(\d,K)>0$ is determined in Lemma \ref{key} and $\kappa=\kappa(D,\|\phi\|_{C^0})>0$ is determined in later arguments (if $\lambda$ does not satisfy any of the above four conditions, then the desired estimate already holds). We may assume that $t_0>0$. As for the term of $\cL \tilde{G}$ we already observe in the proof of Lemma \ref{secdl} that
\[
\cL \log \tlambda_1 \leq C_1 \cF+\frac{1}{\lambda_1} F^{l \bar{k},s\bar{r}} \nabla_{\bar{1}} \chi_{l \bar{k}} \nabla_1 \chi_{s \bar{r}}+F^{k\bar{k}} \frac{|\nabla_{\bar{k}} \tlambda_1|^2}{\lambda_1^2},
\]
where we used the lower bound $\lambda_1 \geq 1$ to obtain the first term.
For the second term of $\cL \tilde{G}$ we compute
\begin{eqnarray*}
\cL \big( \Phi \cdot (|\nabla \phi|^2) \big) &=& \Phi' \cdot \cL(|\nabla \phi|^2)-\Phi'' \cdot F^{q \bar{q}} |\nabla_{\bar{q}} (|\nabla \phi|^2)|^2\\
&=& \Phi' \cdot \big( \nabla^j \phi \cdot \cL(\nabla_j \phi)+\nabla^{\bar{j}} \phi \cdot \cL(\nabla_{\bar{j}} \phi)-F^{q \bar{q}}(|\nabla_q \nabla \phi|^2+|\nabla_q \bnabla \phi|^2) \big)\\
&-& \Phi'' \cdot F^{q \bar{q}} |\nabla_{\bar{q}} (|\nabla \phi|^2)|^2,
\end{eqnarray*}
\[
\nabla_{\bar{j}} \ddt \phi=F^{k\bar{k}} \nabla_{\bar{j}} \chi_{k\bar{k}}.
\]
It follows that
\begin{eqnarray*}
\cL(\nabla_{\bar{j}} \phi)&=& F^{k\bar{k}}(\nabla_{\bar{j}} \chi_{k\bar{k}}-\nabla_k \nabla_{\bar{k}} \nabla_{\bar{j}} \phi) \\
&=& F^{k\bar{k}}(\nabla_{\bar{j}} \hat{\chi}_{k\bar{k}}-R_{k \bar{j}}{}^{\bar{m}}{}_{\bar{k}} \nabla_{\bar{m}} \phi).
\end{eqnarray*}
Using $(4P)^{-1}<\Phi'<(2P)^{-1}$ we get
\[
\Phi' \cdot \nabla^{\bar{j}} \phi \cdot \cL (\nabla_{\bar{j}} \phi) \leq C_2 \cF
\]
for some constant $C_2>0$ depending only on $\a$ and $\hat{\chi}$. The similar estimate also holds for  $\Phi' \cdot \nabla^j \phi \cdot \cL (\nabla_j \phi)$. Thus
\[
\cL \Phi(|\nabla \phi|^2) \leq C_3 \cF-F^{q \bar{q}}(|\nabla_q \nabla \phi|^2+|\nabla_q \bnabla \phi|^2)-\Phi'' \cdot F^{q \bar{q}} |\nabla_{\bar{q}} (|\nabla \phi|^2)|^2.
\]
The estimate for the last term of $\cL \tilde{G}$ is straightforward;
\[
\cL (\Psi(\phi))=\Psi' \cL \phi-\Psi'' F^{k\bar{k}} |\nabla_{\bar{k}} \phi|^2.
\]
Summarizing the above estimates we get
\begin{eqnarray} \label{estmt}
0 \leq \cL \tilde{G} &\leq& F^{k\bar{k}} \frac{|\nabla_{\bar{k}} \tlambda_1|^2}{\lambda_1^2}+\frac{1}{\lambda_1} F^{l \bar{k},s\bar{r}} \nabla_{\bar{1}} \chi_{l \bar{k}} \nabla_1 \chi_{s \bar{r}}+C_4 \cF \nonumber \\
&-& F^{q\bar{q}}(|\nabla_q \nabla \phi|^2+|\nabla_q \bnabla \phi|^2)-\Phi'' \cdot F^{q \bar{q}} |\nabla_{\bar{q}} (|\nabla \phi|^2)|^2 \nonumber \\
&+& \Psi' \cL \phi-\Psi'' F^{k\bar{k}} |\nabla_{\bar{k}} \phi|^2.
\end{eqnarray}
To deal with the first bad term, we will use the second, fifth and last good terms. Set
\[
I:=\{i|F^{i\bar{i}} > \kappa^{-1}F^{1\bar{1}} \}.
\]
We note that $1 \notin I$ since $\kappa<1$. Then at the maximum point we have $\nabla \tilde{G}=0$, which yields that
\begin{eqnarray*}
\sum_{k \notin I} F^{k\bar{k}} \frac{|\nabla_{\bar{k}} \tlambda_1|^2}{\lambda_1^2} &=& \sum_{k \notin I} F^{k\bar{k}} |\Phi' \nabla_{\bar{k}} (|\nabla \phi|^2)+\Psi' \nabla_{\bar{k}} \phi|^2 \\
&\leq& 2(\Phi')^2 \sum_{k \notin I} F^{k\bar{k}}|\nabla_{\bar{k}}(|\nabla \phi|^2)|^2+2(\Psi')^2 \sum_{k \notin I} F^{k\bar{k}} |\nabla_{\bar{k}} \phi|^2 \\
&\leq& \Phi'' \sum_{k \notin I} F^{k\bar{k}}|\nabla_{\bar{k}}(|\nabla \phi|^2)|^2+2(\Psi')^2 \kappa^{-1} F^{1\bar{1}}P.
\end{eqnarray*}
On the other hand,
\[
2 \kappa \sum_{k \in I} F^{k\bar{k}} \frac{|\nabla_{\bar{k}} \tlambda_1|^2}{\lambda_1^2} \leq 2\kappa \Phi'' \sum_{k \in I} F^{k\bar{k}}|\nabla_{\bar{k}}(|\nabla \phi|^2)|^2+4 \kappa (\Psi')^2 \sum_{k \in I} F^{k\bar{k}}|\nabla_{\bar{k}} \phi|^2.
\]
Choose $\kappa=\kappa(D, \|\phi\|_{C^0})$ sufficiently small so that $4 \kappa (\Psi')^2 \leq \frac{1}{2} \Psi''$. Then
\begin{eqnarray} \label{estme}
0 &\leq& \frac{1}{\lambda_1} F^{l \bar{k},s\bar{r}} \nabla_{\bar{1}} \chi_{l \bar{k}} \nabla_1 \chi_{s \bar{r}}+(1-2\kappa) \sum_{k \in I} F^{k\bar{k}}\frac{|\nabla_{\bar{k}} \tlambda_1|^2}{\lambda_1^2}-F^{q\bar{q}}(|\nabla_q \nabla \phi|^2+|\nabla_q \bnabla \phi|^2) \nonumber \\
&+& \Psi' \cL \phi+2(\Psi')^2 \kappa^{-1} F^{1\bar{1}}P+C_4 \cF.
\end{eqnarray}
We note that $\nabla_{\bar{1}} \chi_{1\bar{k}}=\nabla_{\bar{k}} \chi_{1\bar{1}}=\nabla_{\bar{k}} \tlambda_1$ since $d \chi=0$ and we are working at a point in normal coordinates. By the concavity and symmetry of $f$ we have
\[
F^{\ell \bar{k},s \bar{r}} \nabla_{\bar{1}} \chi_{\ell \bar{k}} \nabla_1 \chi_{s \bar{r}} \leq \sum_{k \in I} \frac{F^{1\bar{1}}-F^{k\bar{k}}}{\lambda_1-\lambda_k} |\nabla_1 \chi_{k \bar{1}}|^2=\sum_{k \in I} \frac{F^{1\bar{1}}-F^{k\bar{k}}}{\lambda_1-\lambda_k} |\nabla_{\bar{k}} \tlambda_1|^2
\]
since $\frac{F^{1\bar{1}}-F^{k\bar{k}}}{\lambda_1-\lambda_k} \leq 0$ (\cf \cite[equation (67)]{Sze18}). We know that
\[
\frac{1-\kappa}{\lambda_1-\lambda_k} \geq \frac{1-2\kappa}{\lambda_1}.
\]
Indeed,
\[
(1-\kappa)\lambda_1-(1-2\kappa)(\lambda_1-\lambda_k)=\kappa \lambda_1+(1-2\kappa)\lambda_k.
\]
This expression is clearly positive when $\lambda_k \geq 0$. Otherwise, we have $k=n$, and by using the assumption $\kappa \lambda_1 \geq -\lambda_n$, we get
\[
\kappa \lambda_1+(1-2\kappa)\lambda_n \geq -2\kappa \lambda_n>0.
\]
Thus we have
\[
\sum_{k \in I} \frac{F^{1\bar{1}}-F^{k\bar{k}}}{\lambda_1-\lambda_k}|\nabla_{\bar{k}} \tlambda_1|^2 \leq -\sum_{k \in I} \frac{(1-\kappa)F^{k\bar{k}}}{\lambda_1-\lambda_k} |\nabla_{\bar{k}} \tlambda_1|^2 \leq -\frac{1-2\kappa}{\lambda_1} \sum_{k \in I} F^{k\bar{k}}|\nabla_{\bar{k}} \tlambda_1|^2.
\]
So the first and second terms of \eqref{estme} are estimated as
\[
\frac{1}{\lambda_1} F^{l \bar{k},s\bar{r}} \nabla_{\bar{1}} \chi_{l \bar{k}} \nabla_1 \chi_{s \bar{r}}+(1-2\kappa) \sum_{k \in I} F^{k\bar{k}}\frac{|\nabla_{\bar{k}} \tlambda_1|^2}{\lambda_1^2} \leq 0
\]
since $\lambda_1 \geq 1$. As for the forth good term of \eqref{estme}, we use the following estimate
\[
F^{q\bar{q}}(|\nabla_q \nabla \phi|^2+|\nabla_q \bnabla \phi|^2) \geq F^{1\bar{1}}|\lambda_1-\hat{\chi}_{1\bar{1}}|^2 \geq F^{1\bar{1}}\frac{\lambda_1^2}{2}-C_5 \cF
\]
for some constant $C_5$ depending only on $\a$ and $\hat{\chi}$. Putting all things together we obtain
\[
0 \leq F^{11} \bigg(2(\Psi')^2 \kappa^{-1}P-\frac{\lambda_1^2}{2} \bigg)+\Psi' \cL \phi+C_6 \cF
\]
with $C_6:=C_4+C_5$. Now we invoke Lemma \ref{key}. From the assumption $\frac{1}{1+\lambda_1^2} \leq \frac{\r}{1+\lambda_n^2}$, we observe that $\frac{1}{1+\lambda_1^2}<\r \sum_{i=1}^n \frac{1}{1+\lambda_i^2}$, or equivalently $F^{11}<\r \cF$. So we have $\cL \phi \geq \r \cF$. Since $\Psi'<0$, the above inequality yields that
\[
0 \leq F^{11} \bigg(2(\Psi')^2 \kappa^{-1}P-\frac{\lambda_1^2}{2} \bigg)+(\r \Psi'+C_6)\cF.
\]
We take $D>0$ sufficiently large so that $\r \Psi'+C_6<0$ (this is possible since the constant $C_6$ does not depend on $\kappa$). Then we have
\[
\frac{\lambda_1^2}{2} \leq 2(\Psi')^2 \kappa^{-1}P.
\]
This yields the desired bound.
\end{proof}
\begin{lem} \label{gradb}
There is a constant $C>0$ depending on $\a$, $\hat{\chi}$, $\hat{\Theta}$ and $\phi_0$ such that
\begin{equation} \label{gradbd}
\sup_{X \times [0,T)} |\nabla \phi_t|_\a^2 \leq C.
\end{equation}
\end{lem}
In \cite[Lemma 4]{PT17}, they give the gradient estimate like Lemma \ref{gradb} by the blowup argument combined with Sz\'ekelyhidi's Liouville theorem for $\Gamma$-solutions (\cf \cite[Section 5]{Sze18}). This argument does not apply to the TLPF due to the lack of the structural properties for $f$ as mentioned in the beginning of this subsection. However, since we have a uniform lower bound for the eigenvalues in our case by Corollary \ref{unifev}, the argument is rather simple. We follow closely to the argument \cite[Proposition 5.1]{CJY15};
\begin{proof}[Proof of Lemma \ref{gradb}]
Assume that \eqref{gradbd} does not hold. Then there exists a sequence $(x_k,t_k) \in X \times [0,T)$ with $t_k \to T$ such that
\[
C_k:=|\nabla \phi(x_k,t_k)|_\a=\sup_{X \times [0,t_k]}|\nabla \phi|_\a \to \infty
\]
as $k \to \infty$. By passing to a subsequence, we may further assume that $\{x_k\}$ converges to some point $x \in X$. From the previous arguments, there is a uniform constant $C>0$ such that
\begin{itemize}
\item $\hat{\chi}+\dd \phi_t \geq -C\a$ on $X \times [0,T)$.
\item $\sup_{X \times [0,T)}|\phi_t| \leq C$.
\item $|\dd \phi(x,t_k)|_\a \leq C(1+\sup_{X \times [0,t_k]}|\nabla \phi|_\a^2)$ for all $x \in X$ and $k$.
\end{itemize}
For each $k$, we take a local coordinates $\{U_k,(z_1,\ldots,z_n)\}$ centered at $x_k$, identifying with the ball $B_1(0)$ of radius $1$, where $\a=\Id+O(|z|^2)$. By replacing $C$ by a slightly large number, we may further assume that $\a$ is the Euclidean metric. We define $\phi_k(z):=\phi(z/C_k,t_k)$ defined on the ball of radius $C_k$. Then we have;
\begin{itemize}
\item $\dd \phi_k \geq (-C\Id-\hat{\chi})/C_k^2$ on $B_{C_k}(0)$.
\item $\sup_{B_{C_k}(0)} |\phi_k| \leq C$.
\item $|\dd \phi_k|_\a \leq 2C$ on $B_{C_k}(0)$.
\item $|\nabla \phi_k|_\a \leq 1=|\nabla \phi_k(0)|_\a$ on $B_{C_k}(0)$.
\end{itemize}
The proof can now be completed exactly as in \cite[Proposition 5.1]{CJY15}. So for a fixed $\b \in (0,1)$, we know that by passing to a subsequence, $\phi_k$ converges to $\phi_\infty \colon \C^n \to \R$ in $C_{\rm loc}^{1,\b}$ as $k \to \infty$. Moreover, the function $\phi_\infty$ is continuous, uniformly bounded, has $|\nabla \phi_\infty(0)|_\a=1$ and satisfies $\dd \phi_{\phi_\infty} \geq 0$ in the sense of distributions. Such a function must be a constant (\cf \cite{Ron74}), which contradicts to $|\nabla \phi_\infty(0)|_\a=1$.
\end{proof}
\subsection{Convergence of the flow}
Now we will finish the proof of Theorem \ref{convf} by showing that;
\begin{thm} \label{convpt}
Let $X$ be a compact complex manifold with a K\"ahler form $\a$, and $\hat{\chi}$ a closed real $(1,1)$-form. Assume that $\hat{\Theta} \in ((n-1)\hpi,n \hpi)$ and there is a $C$-subsolution. Then the tangent Lagrangian phase flow $\phi_t$ starting from any potential $\phi_0 \in \cH$ converges to the deformed Hermitian Yang-Mills metric $\phi_\infty \in \cH$ in the $C^\infty$-topology.
\end{thm}
We can show this by using the argument in \cite[Lemma 7]{PT17}.
\begin{proof}[Proof of Theorem \ref{convpt}]
From the previous subsection, Lemma \ref{evkry} and a standard bootstrapping argument we obtain a uniform $C^k$ control along the flow $\phi_t$ for any non-negative integer $k$. To prove the $C^\infty$-convergence, we set $\psi_t:=\ddt \phi_t+A$ for some uniform constant $A>0$ such that $\psi_t>0$ for all $t \in [0,\infty)$ by using Lemma \ref{sff}. Then $\psi$ satisfies the same heat equation as $\ddt \phi$;
\begin{equation} \label{heat}
\ddt \psi=F^{i \bar{j}} \p_i \p_{\bar{j}} \psi.
\end{equation}
                                                                                                                                                                                                                                                                                                                                                                                                                                                                                                                                                                                                                                                                                                                                                                                                                                                                                                                                                                                                                                                                                                                                                                                                                                                                                                                                                                                                                                                                                                                                                                                                                                                                                          Since we already know that the RHS of \eqref{heat} is uniformly elliptic by the $C^2$-estimate, we can apply the differential Harnack inequality on compact Hermitian manifolds to \eqref{heat}, and obtain
\[
\osc_X \ddt \phi (\cdot,t)=\osc_X \psi (\cdot,t) \leq C_1 e^{-C_2t}
\]
for some constants $C_1, C_2>0$ (see \cite[Section 6, Section 7]{Gil11} for more details). On the other hand, in the proof of Proposition \ref{eff} we observe that
\[
\int_X \ddt \phi \cdot \Re \big( e^{-\sqrt{-1}\hat{\Theta}} (\a+\sqrt{-1}\chi_\phi)^n \big)=0,
\]
which in particular shows that there exists a point $y=y(t) \in X$ such that $\ddt \phi(y,t)=0$ since the measure $\Re \big( e^{-\sqrt{-1}\hat{\Theta}} (\a+\sqrt{-1}\chi_\phi)^n \big)$ is positive along the flow. Therefore for any $x \in X$ we have
\[
\bigg|\ddt \phi(x,t) \bigg|=\bigg| \ddt \phi (x,t)-\ddt \phi(y,t) \bigg| \leq \osc_X \ddt \phi (\cdot,t) \leq C_1 e^{-C_2 t},
\]
and hence
\[
\ddt \bigg( \phi_t+\frac{C_1}{C_2} e^{-C_2 t} \bigg)=\ddt \phi-C_1 e^{-C_2 t} \leq 0.
\]
So the function $\phi_t+\frac{C_1}{C_2} e^{-C_2 t}$ is decreasing in $t$, and uniformly bounded by the $C^0$ estimate. Thus it converges to a function $\phi_\infty$. By the higher order estimates, we know that this convergence is actually in $C^\infty$. The function $\phi_t$ also converges to the same function $\phi_\infty$ in $C^\infty$. The convergence $\ddt \phi \to 0$ yields that the function $\phi_\infty$ must satisfies the equation $F(A[\phi_\infty])=0$, so we have $\Theta(A[\phi_\infty])=\hat{\Theta}$. This completes the proof.                                                                                                                                                                                                                                                                                                                                                                                                                                                                                                                                                                                                                                                                                                                                                                                                                                                                                                                                                                                                                                                                                                                                                                                                                                                                                                                                                                                                                                                                                                                                                                                                                                                                                                                                                                                                                                                                                                                                                                                                                                                                                                                                                                                                                                                                                                                                                                                                                                                                                                                                                                                                                                                                                                                                                                                                                                                                                                                                                                                                                                                                                                                                                                                                                                  
\end{proof}
\begin{rk}
Using the monotonicity of $\cC$ and $\cV$ (\cf Proposition \ref{eff}) together with the uniqueness result of dHYM metrics \cite[Theorem 1.1]{JY17}, one can easily obtain an alternative proof of the $C^\infty$-convergence of the TLPF in the same way as in the proof of \cite[Theorem 1.1]{Tak19}.
\end{rk}


\newpage
\begin{thebibliography}{widestlabel}

\bibitem[CCL20]{CCL20}T.~C.~Collins, J.~Chu and M.~C.~Lee:
	\newblock \emph{The space of almost calibrated $(1,1)$ forms on a compact K\"ahler manifold}.
	\newblock \texttt{arXiv:2002.01922}.

\bibitem[Che19]{Che19}G.~Chen:
	\newblock \emph{On $J$-equation}.
	\newblock \texttt{arXiv:1905.10222}.

\bibitem[CHT17]{CHT17}T.~C.~Collins, T.~Hisamoto and R.~Takahashi:  
\newblock\emph{The inverse Monge-Amp\`ere flow and applications to K\"ahler-Einstein metrics}. 
\newblock \texttt{arXiv:1712.01685}, to appear in J. Diff. Geom.


\bibitem[CJY15]{CJY15}T.~C.~Collins, A.~Jacob and S.~T.~Yau:
	\newblock \emph{$(1,1)$-forms with specified Lagrangian phase: a priori estimates and algebraic obstructions}.
	\newblock \texttt{arXiv:1508.01934}.

\bibitem[CPW17]{CPW17}T.~C.~Collins, S.~Picard and X.~Wu:
	\newblock \emph{Concavity of the Lagrangian phase operator and applications}.
	\newblock Calc. Var. and PDE \textbf{56} (2017), no. 4, Art. 89.

\bibitem[CXY17]{CXY17}T.~C.~Collins, D.~Xie and S.~T.~Yau:
	\newblock \emph{The deformed Hermitian-Yang-Mills equation in geometry and physics}.
	\newblock \texttt{arXiv:1712.00893}.

\bibitem[CY18]{CY18}T.~C.~Collins and S.~T.~Yau:
	\newblock \emph{Moment maps, nonlinear PDE, and stability in mirror symmetry}.
	\newblock \texttt{arXiv:1811.04824}.

\bibitem[DS16]{DS16}R.~Dervan and G.~Sz\'ekelyhidi:
	\newblock \emph{The K\"ahler-Ricci flow and optimal degenerations}.
	\newblock \texttt{arXiv:1612.07299}.

\bibitem[Gil11]{Gil11}M.~Gill:
	\newblock \emph{Convergence of the parabolic complex Monge-Amp\`ere equation on compact Hermitian manifolds}.
	\newblock Comm. Anal. Geom. \textbf{19} (2011), no. 2, 277-303.

\bibitem[His19]{His19}T.~Hisamoto:
	\newblock \emph{Geometric flow, multiplier ideal sheaves and optimal destabilizer for a Fano manifold}.
	\newblock \texttt{arXiv:1901.08480}.

\bibitem[HJ20]{HJ20}X.~Han and X.~Jin:
	\newblock \emph{Stability of line bundle mean curvature flow}.
	\newblock \texttt{arXiv:2001.07406}.

\bibitem[HL20]{HL20}F.~R.~Harvey and H.~B.~ Lawson, Jr.:
	\newblock \emph{Pseudoconvexity for the Special Lagrangian Potential Equation}.
	\newblock \texttt{arXiv:2001.09818}.

\bibitem[HY19]{HY19}X.~Han and H.~Yamamoto:
	\newblock \emph{An $\e$-regularity theorem for line bundle mean curvature flow}.
	\newblock \texttt{arXiv:1904.02391}.

\bibitem[JY17]{JY17}A.~Jacob and S.~T.~Yau:
	\newblock \emph{A special Lagrangian type equation for holomorphic line bundles}.
	\newblock Math. Ann., \textbf{369} (2017), 869--898.

\bibitem[Kry76]{Kry76} N.~V.~Krylov: 
	\newblock \emph{Sequences of convex functions, and estimates of the maximum of the solution of a parabolic equation}.
	\newblock Sibirsk. Mat. \u{Z}. {\bf 17} (1976), no. 2, 290--303, 478.

\bibitem[Kry82]{Kry82}N.~V.~Krylov:
	\newblock \emph{Boundedly inhomogeneous elliptic and parabolic equations}.
	\newblock Izv. Akad. Nauk SSSR Ser. Mat., \textbf{46} (1982), no. 3, 487--523.
	
\bibitem[LYZ01]{LYZ01}C.~Leung, S.~T.~Yau and E.~Zaslow:
	\newblock \emph{From special Lagrangian to Hermitian Yang-Mills via Fourier-Mukai transform}.
	\newblock Winter School on Mirror Symmetry, Vector Bundles and Lagrangian Submanifolds (Cambridge, MA, 1999), 209--225, AMS.IP Stud. Adv. Math., \textbf{23}, Amer. Math. Soc., Providence, RI, 2001.

\bibitem[MMMS00]{MMMS00}M.~Mari\~no, R.~Minasian, G.~Moore and A.~Strominger:
	\newblock \emph{Nonlinear instantons from supersymmetric p-branes}.
	\newblock Izv. J.~High Energy Phys., (2000), no. 1.

\bibitem[Nev13]{Nev13}A.~Neves:
	\newblock \emph{Finite time singularities for Lagrangian mean curvature flow}.
	\newblock Ann. Math. {\bf 177} (2013), 1029--1076.

\bibitem[Pin19]{Pin19}V.~P.~Pingali:
	\newblock \emph{The deformed Hermitian Yang-Mills equation on Three-folds}.
	\newblock \texttt{arXiv:1910.01870}.

\bibitem[PT17]{PT17}D.~H.~Phong and D.~T.~T\^{o}:
	\newblock \emph{Fully non-linear parabolic equations on compact hermitian manifolds}.
	\newblock \texttt{arXiv:1711.10697}.

\bibitem[Ron74]{Ron74}L.~I.~Ronkin:
	\newblock \emph{Introduction to the theory of entire functions of several variables}.
	\newblock Transaction of Mathematical Monographs, \textbf{44}, Amer. Math. Soc., Providence, R.I., 1974.

\bibitem[Sj\"o19]{Sjo19}Z.~Sj\"ostr\"om Dyrefelt:
	\newblock \emph{Optimal lower bounds for Donaldson's $\cJ$-functional}.
	\newblock \texttt{arXiv:1907.01486}.

\bibitem[Sol13]{Sol13}J.~P.~Solomon:
	\newblock \emph{The Calabi homomorphism, Lagrangian paths and special Lagrangians}.
	\newblock Math. Ann., \textbf{357} (2013), no. 4, 1389--1424.

\bibitem[SS19]{SS19}E.~Schlitzer and J.~Stoppa:
	\newblock \emph{Deformed Hermitian Yang-Mills connections, extended Gauge group and scalar curvature}.
	\newblock \texttt{arXiv:1911.10852}.

\bibitem[Sz\'e18]{Sze18}G.~Sz\'ekelyhidi:
	\newblock \emph{Fully non-linear elliptic equations on compact hermitian manifolds}.
	\newblock J. Differential Geom., \textbf{109} (2018), no. 2, 337--378.

\bibitem[Tak19]{Tak19}R.~Takahashi:
	\newblock \emph{Collapsing of the line bundle mean curvature flow on K\"ahler surfaces}.
	\newblock \texttt{arXiv:1912.13145}.

\bibitem[Tho01]{Tho01}R.~P.~Thomas:
	\newblock \emph{Moment maps, monodromy, and mirror manifolds}.
	\newblock Symplectic geometry and mirror symmetry (Seoul, 2000), 467--498, World Sci. Publ., River Edge, NJ, 2001.

\bibitem[Tso85]{Tso85}K.~Tso:
	\newblock \emph{On an Aleksandrov-Bakel'man type maximum principle for second-order parabolic equations}.
	\newblock Comm. Par. Diff. Eq., \textbf{10} (1985), no. 5, 543--553.

\bibitem[TY02]{TY02}R.~P.~Thomas and S.~T.~Yau:
	\newblock \emph{Special Lagrangians, stable bundles and mean curvature
flow}.
	\newblock Comm. Anal. Geom., \textbf{10} (2002), 1075--1113.

\bibitem[Wan12]{Wan12} Y.~Wang:
	\newblock \emph{On the $C^{2,\a}$-regularity of the complex Monge-Amp\`ere equation}.
	\newblock Math. Res. Lett. \textbf{19} (2012), 939--946.

\bibitem[Xia19]{Xia19}M.~Xia:
	\newblock \emph{On sharp lower bounds for Calabi type functionals and destabilizing properties of gradient flows}.
	\newblock \texttt{arXiv:1901.07889}.

\bibitem[Yua05]{Yua05} Y.~Yuan:
	\newblock \emph{Global solutions to special Lagrangian equations}.
	\newblock Proc. Amer. Math. Soc., \textbf{134} (2005), 1355--1358.
	
\end{thebibliography}
\end{document}